\documentclass{amsart}
\usepackage{amsmath}
\usepackage{amsfonts,epsfig}
\usepackage{latexsym}
\usepackage{amssymb}
\usepackage{amsthm}
\usepackage{graphicx}
\usepackage{amscd}
\usepackage{cite}

\newtheorem{thm}{Theorem}[section]
\theoremstyle {definition}
\newtheorem{cor}[thm]{Corollary}
\newtheorem{prop}[thm]{Proposition}
\newtheorem{defn}[thm]{Definition}
\newtheorem{lem}[thm]{Lemma}
\newtheorem{eg}[thm]{Example}

\newtheorem{rem}[thm]{Remark}

\numberwithin{equation}{section}

\begin{document}
\
\title[On S-primary submodules ]
{On S-primary submodules}%

\author{H. Ansari-Toroghy }%
\address{Department of pure Mathematics,
Faculty of mathematical Sciences, University of Guilan,
P. O. Box 41335-19141 Rasht, Iran.} %
\email{ansari@guilan.ac.ir}%

\author{S. S. Pourmortazavi}%
\address{Department of pure Mathematics,
Faculty of mathematical Sciences, University of Guilan,
P. O. Box 41335-19141 Rasht, Iran.} %
\email{mortazavi@phd.guilan.ac.ir}%


\date{\today}%
\begin{abstract}
             Let $R$ be a commutative ring with identity, $S$ a multiplicatively closed subset of $R$, and $M$ be an $R$-module.
             In this paper, we study and investigate some properties of $S$-primary submodules of $M$. Among the other results, it is shown that this
             class of modules contains the family of primary (resp. $S$-prime) submodules properly.
 \end{abstract}

\maketitle


                                                       \section{Introduction}

              Throughout this article, all rings are commutative with identity elements, and all modules
              are unital left modules. $\mathbb{N}$, $\mathbb{Z}$, and $\mathbb{Q}$ will denote respectively the natural
              numbers, the ring of integers and the field of quotients of $\mathbb{Z}$.

              Consider a non-empty subset $S$ of $R$. We call $S$ a multiplicatively closed subset of $R$ if (i) $0 \notin S$, (ii)
              $1 \in S$, and (iii) $ss^{\prime} \in S$ for all $s, s^{\prime} \in S$ \cite{wang2016foundations}. Note that $S=R-p$
              is a multiplicatively closed subset of $R$ for every prime ideal $p$ of $R$. Let $N$ and $K$ be two
              submodules of an $R$-module $M$ and $J$ an ideal of $R$. Then the residual $N$ by $K$ and $J$ is defined as follows:
              $$(N:_{R}K)=\{ \, r \in R\, \, |\,\, rK \subseteq N\},$$
              $$(N:_{M}J)=\{\, m \in M \,\,|\,\, Jm \subseteq N\}.$$

              Particularly, we use $Ann_{R}(M)$ instead of $(0:_{R}M)$, and we use $(N:_{M}s)$ instead of $(N:_{M}Rs)$, where $Rs$ is the principal ideal generated by an element $s \in R$. The sets of prime ideals and maximal ideals of $R$ are denoted by $Spec(R)$ and $Max(R)$, respectively.

              A submodule $P$ of $M$ is called \textit{prime} if $P\neq M$ and whenever $r\in R$ and $e \in M$ satisfy $re\in P$, then $r\in (P:_{R}M)$ or $e\in P$. The set of all prime submodule of $M$ is denoted by $Spec(M)$(See \cite{dauns1978prime, lu2007module}).

              In \cite{sevim2019s}, the authors introduce the concept of $S$-prime submodule and investigate some properties of this class of modules. Let $S$ be a multiplicatively closed subset of $R$ and $P$ be a submodule of $M$ with $(P:_{R}M) \cap S=\emptyset$. Then $P$ is said to be an $S$-\textit{prime submodule} if there exists $s \in S$ such that whenever $rm \in P$, where $r \in R$ and $m \in M$, then $sr \in (P:_{R}M)$ or $sm \in P$. Particularly, an ideal $I$ of $R$ is said to be an $S$-\textit{prime ideal} if $I$ is an $S$-prime submodule of the $R$-module $R$.

              The notion of $S$-primary submodule was introduced in \cite{farshadifararxiv2003}. Let $S$ be a multiplicatively closed subset of $R$ and $P$ be a submodule of $M$ with $(P:_{R}M) \cap S=\emptyset$.  Then $P$ is said to be an $S$-\textit{primary submodule} if there exists $s \in S$ such that whenever $rm \in P$, where $r \in R$ and $m \in M$, then $sr \in \sqrt{(P:_{R}M)}$ or $sm \in P$. The author has not specified the properties of this family of modules excepts a few results about a certain $R$-module.

              In this paper, we will study the family of $S$-primary submodules extensively and investigate some of their properties. In fact, this family of modules is a generalization of primary (resp. $S$-prime) submodules.

              Among the other results, we provide some notions that each one is equivalent to $S$-primary (see Theorem \ref{1}). Examples \ref{2.2} and \ref{2.3} show that these new modules contain the family of primary and $S$-prime submodules properly. Further It is proved that if $P$ is an $S$-primary submodule of $M$, then $S^{-1}P$ is also an $S$-primary submodule of $S^{-1}M$ (see  Proposition \ref{4}). Example \ref{6} shows that the converse is not true in general. Also we show that $S$-primary submodules has a good behavior with direct sums (see Theorems \ref{17} and \ref{18}). Moreover, we provide some useful characterization concerning $S$-primary submodules (see Theorems \ref{14}, \ref{21}, and \ref{22}).

                                                             \section{Main results}

\begin{defn}\label{0.1}
        Let $S$ be a multiplicatively closed subset of $R$ and $P$ be a submodule of $M$ with $(P:_{R}M) \cap S=\emptyset$. Then $P$ is said to be an $S$-\textit{primary submodule} of $M$ if there exists $s \in S$ such that whenever $rm\in P$, where $m \in M$ and $r \in R$, then $sr \in
        \sqrt{(P:_{R}M)}$ or $sm \in P$ \cite[Definition 2.27 ]{farshadifararxiv2003}. In particular, we say that an ideal $I$ of $R$ is an $S$-primary ideal if $I$ is an $S$-primary submodule of $R$-module $R$.
\end{defn}

 \begin{thm}\label{1}

                            Let $S$ be a multiplicatively closed subset of $R$. For a submodule $P$ of an $R$-module $M$
                            with $(P:_{R}M)\cap S=\emptyset$. Then the following are equivalent:

 \begin{itemize}

          \item [(a)] $P$ is an $S$-primary submodule of $M$;

          \item [(b)] There exists $s \in S$ such that for every $r\in R$, the endomorphism\\ $r:s(M/P) \rightarrow s(M/P)$ given by
                           $s\overline{m}=sm+P\mapsto rs\overline{m}=rsm+P$
                      is injective or $(rs)^{t}(M/P)=(\overline{0})$ for some $t \in \mathbb{N}$ ;

          \item [(c)] There exists $s \in S$ such that whenever $rN\subseteq P$, where $N$ is a submodule of $M$ and $r \in R$, then $sr \in
                      \sqrt{(P:_{R}M)}$ or $sN \subseteq P$;

          \item [(d)] There exists $s \in S$ such that whenever $JN\subseteq P$, where $N$ is a submodule of $M$ and $J$ is an ideal of $R$, then $sJ
                      \subseteq \sqrt{(P:_{R}M)}$ or $sN \subseteq P$.

\end{itemize}

\begin{proof}

    (a)$\Rightarrow$(b). By hypothesis, there exists $s \in S$ such that for every $r \in R$ and $m \in M$ if $rm \in P$, then we have $sm \in P$ or $sr
                         \in \sqrt{(P:_{R}M)}$. Now for each $r \in R$, we define the endomorphism $r: s(M/P) \rightarrow s(M/P)$ by $sm+P\mapsto rsm+P$. We show that this endomorphism is injective or $rs \in \sqrt{(P:_{R}M)}$. Assume $rs \notin \sqrt{(P:_{R}M)}$. Then we show the other part holds. To see let $\overline{rsm}=rsm+P=P=\overline{0}$. So we have $(rs)m \in P$. So by hypothesis, $sm \in P$ or $s(rs)=rs^{2} \in \sqrt{(P:_{R}M)}$. We conclude $sr \in \sqrt{(P:_{R}M)}$ which is a contradiction.
                         Hence $sm \in P$, as required.

    (b)$\Rightarrow$(a). It is clear.

    (a)$\Rightarrow$(c). It is clear.                     	

	(c)$\Rightarrow$(d). Let $JN \subseteq P$, where $J$ is an ideal of $R$ and $N$ is a submodule of $M$. We will show that there exists $s \in S$
                         such that $sN \subseteq P$ or $sJ \subseteq \sqrt{(P:_{R}M)}$. Clearly, we have $rN \subseteq P$ for every $r \in J$. So by part (c), there exists $s \in S$ such that $sN \subseteq P$ or $sr \in \sqrt{(P:_{R}M)}$ for every $r \in J$, as desired.

	(d)$\Rightarrow$(a). Take $r \in R$ and $m \in M$ with $rm \in P$. Now, put $J=Rr$ and $N=Rm$. Then we can conclude that $JN=Rrm \subseteq P$.
	                     By assumption, there is an $s \in S$ so that $sJ=Rrs \subseteq \sqrt{(P:_{R}M)}$ or $sN=Rsm \subseteq P$ and so either $sr \in
                         \sqrt{(P:_{R}M)}$ or $sm \in P$, as required.

\end{proof}

\end{thm}

\begin{lem}\label{0.3}
     Let $M$ be an $R$-module and $S$ a multiplicatively closed subset of $R$. Then we have the following.

     \begin{itemize}

       \item [(a)] If $P$ is a primary submodule of $M$ such that $(P:_{R}M) \cap S=\emptyset$, then $P$ is an $S$-primary submodule of $M$.

       \item [(b)] If $P$ is an $S$-primary submodule of $M$ and $S \subseteq u(R)$, where $u(R)$ denotes the set of units in $R$, then $P$ is a
                   primary submodule of $M$.

     \end{itemize}

\end{lem}
 \begin{proof}
   This is clear.
\end{proof}

By setting $S=\{1\}$, we conclude that every primary submodule is an $S$-primary submodule by Lemma \ref{0.3}. The following example shows that the converse is not true in general.

\begin{eg}\label{2.2}

        Consider the $\mathbb{Z}$-module $M=\mathbb{Q}\oplus (\bigoplus_{i=1}^{n} \mathbb{Z}_{p_{i}})$, where $p_{i}$ are distinct positive prime integers. Take the submodule $P=(0)$ and the multiplicatively closed subset
        $$S=\{1, p_{1}^{m_{1}}p_{2}^{m_{2}}... p_{n}^{m_{n}}\,|\, \forall i \in \{1,2,...,n\},\,\, m_{i} \in \mathbb{N} \cup \{0\}\,\}.$$
        First note that $(P:_{\mathbb{Z}}M)=(0)$ and $p_{1}p_{2}...p_{n}(0,\overline{1},\overline{1},...,\overline{1})=(0,\overline{0},\overline{0},...,\overline{0}) \in P$. Since $p_{1}p_{2}...p_{n} \notin \sqrt{(P:_{\mathbb{Z}}M)}$ and $(0,\overline{1},\overline{1},...,\overline{1})\notin P$, $P$ is not a primary submodule of $M$. Put $s=p_{1}p_{2}...p_{n}$ and let
        $$k(\frac{p}{q},\overline{x_{1}}, \overline{x_{2}},..., \overline{x_{n}})= (\frac{kp}{q},\overline{kx_{1}}, \overline{kx_{2}},..., \overline{kx_{n}}) \in P,$$
        where $k \in \mathbb{Z}$ and $(\frac{p}{q},\overline{x_{1}}, \overline{x_{2}},..., \overline{x_{n}}) \in M$. Then $kp=0$. This yields that $k=0$ or $p=0$. If $k=0$, there is nothing to prove. Thus assume that $p=0$. Then $s(\frac{p}{q},\overline{x_{1}}, \overline{x_{2}},..., \overline{x_{n}}) \in P$. Therefore, $P$ is an $S$-primary submodule of $M$.

\end{eg}

        We recall that a submodule $P$ of an $R$-module $M$ is $S$-\textit{prime} if there exists $s \in S$ such that whenever $rm \in P$, where $r \in R$ and $m \in M$,  then $sr \in (P:_{R}M)$ or $sm \in P$ \cite{sevim2019s}. Clearly, every $S$-prime submodule is $S$-primary. The following example shows that the converse is not true in general.

\begin{eg}\label{2.3}
     Consider $M=\mathbb{Z}_{4}$ as a $\mathbb{Z}$-module. Set $S=\mathbb{Z}\backslash 2 \mathbb{Z}$ and $P=(\overline{0})$. Thus we have $(P:_{\mathbb{Z}}M)=4\mathbb{Z}$ and $2.\overline{2} \in (\overline{0})$.
     Since for every $s \in S$, $ 2s \notin (P:_{\mathbb{Z}}M)$ and $s.\overline{2} \notin P$, $P$ is not an $S$-prime submodule of $M$. Put $s=1$ and let $k\overline{a}=\overline{0}$. If $\overline{a}=\overline{0}$, there is nothing to prove. Thus assume that $\overline{a}\neq \overline{0}$.   Then $k=2k^{\prime}$ for some $k^{\prime} \in \mathbb{Z}$. This implies that $k \in \sqrt{(P:_{\mathbb{Z}}M)}$. Therefore, $P$ is an $S$-primary submodule of $M$.

\end{eg}

 \begin{rem}\label{3}
                    Let $S$ be a multiplicatively closed subset of $R$. Recall that the saturation $S^{*}$ of $S$ is defined as
		            $$S^{*}=\{x \in R\, \, | \, \, \frac{x}{1} \,\, \text{is a unit of}\,\, S^{-1}R\}.$$
		            It is obvious that $S^{*}$ is a multiplicatively closed subset of $R$ containing $S$ \cite{gilmer1992multiplicative}.
	
\end{rem}

 \begin{prop} \label{4}
            Let $S$ be a multiplicatively closed subset of $R$ and $M$ be an $R$-module. Then we have the following.

 \begin{itemize}

        \item [(a)] If $S_{1} \subseteq S_{2}$ are multiplicatively closed subsets of $R$ and $P$ is an $S_{1}$-primary submodule
                    of $M$, then $P$ is an $S_{2}$-primary submodule of $M$ in case \linebreak
                    $(P:_{R}M) \cap S_{2}=\emptyset$.

        \item [(b)] $P$ is an $S$-primary submodule of $M$ if and only if $P$ is an $S^{*}$-primary submodule of $M$.

        \item [(c)] If $P$ is an $S$-primary submodule of $M$, then $S^{-1}P$ is a primary submodule of $S^{-1}R$-module $S^{-1}M$.

 \end{itemize}

 \end{prop}

 \begin{proof}

 \begin{itemize}
 	
  	\item [(a)] It is clear.

    \item [(b)] Assume that $P$ is an $S$-primary submodule of $M$. We need to prove that $(P:_{R}M)$ and $S^{*}$ are disjoint. Suppose there exists
                $x \in (P:_{R}M) \cap S^{*}$. As $x \in S^{*}$, $\frac{x}{1}$ is a unit of $S^{-1}R$ and so  $(\frac{x}{1})(\frac{a}{s})=1$ for some $a \in R$ and $s \in S$. This yields that $us=uxa$ for some $u \in S$. Now we have that $us=uxa \in (P:_{R}M) \cap S$, a contradiction. Thus $(P:_{R}M) \cap S^{*}=\emptyset$. Now as $S \subseteq S^{*}$, by part (b), $P$ is an $S^{*}$-primary submodule of $M$. Conversely, assume that $P$ is an $S^{*}$-primary submodule of $M$. Let $rm \in P$, where $r \in R$ and $m \in M$. Then there exists $x \in S^{*}$ such that $xr \in \sqrt{(P:_{R}M)}$ or $xm \in P$. As $\frac{x}{1}$ is a unit of $S^{-1}R$, there exist $u,s \in S$ and $a \in R$ such that $us=uxa$. Put $us=s^{\prime} \in S$. Then note that $s^{\prime}r=(us)r=uaxr \in \sqrt{(P:_{R}M)}$ or $s^{\prime}m \in P$. Therefore, $P$ is an $S$-primary submodule of $M$.

    \item [(c)] Let $(\frac{r}{s})(\frac{m}{t}) \in S^{-1}P$, where $\frac{r}{s} \in S^{-1}R$ and
                $\frac{m}{t} \in S^{-1}M$. Then $urm \in P$ for some $u \in S$. Since $P$ is an $S$-primary submodule of $M$, there is an $s^{\prime} \in S$ so that $s^{\prime}ur \in \sqrt{(P:_{R}M)}$ or $s^{\prime}m \in P$. This yields that $\frac{r}{s}=\frac{s^{\prime}ur}{s^{\prime}us} \in  S^{-1}\sqrt{(P:_{R}M)} \subseteq \sqrt{(S^{-1}P:_{S^{-1}R}S^{-1}M)}$ or  $\frac{m}{t}=\frac{s^{\prime}m}{s^{\prime}t} \in S^{-1}P$. Hence, $S^{-1}P$ is a primary submodule of $S^{-1}M$.

\end{itemize}

\end{proof}

The following example shows that the converse of part (c) of Proposition \ref{4} is not true in general.

\begin{eg}\label{6}

        Consider the $\mathbb{Z}$-module $M=\mathbb{Q}$. Take the submodule $N=\mathbb{Z}$ and the multiplicatively closed subset $S=\mathbb{Z}-\{0\}$ of $\mathbb{Z}$. Then $(N:_{\mathbb{Z}}M)=(0)$. Let $s$ be an arbitrary element of $S$. Choose a prime number $p$ with $gcd(p,s)=1$. Then note that $p \frac{1}{p}=1 \in N$. But $sp \notin \sqrt{(N:_{\mathbb{Z}}M)}$ and $\frac{s}{p} \notin N$, it follows that $N$ is not an $S$-primary submodule of $M$. Since $S^{-1}\mathbb{Z}=\mathbb{Q}$ is a field, $S^{-1}(\mathbb{Q})$ is a vector space. Therefore the proper submodule $S^{-1}N$ is a primary submodule of $S^{-1}\mathbb{Q}$.
\end{eg}

  \begin{prop} \label{6.1}

                  Suppose $f:M \rightarrow M^{\prime}$ is an $R$-homomorphism. Then we have the following.

    \begin{itemize}

            \item [(a)] If $P^{\prime}$ is an $S$-primary submodule of $M^{\prime}$ provided that $(f^{-1}(P^{\prime}):_{R}M) \cap S= \emptyset$, then $f^{-1}(P^{\prime})$ is an $S$-primary submodule of $M$.

            \item [(b)] If $f$ is an epimorphism and $P$ is an $S$-primary submodule of $M$ with $ker(f) \subseteq P$, then $f(P)$ is an $S$-primary submodule of $M^{\prime}$.

    \end{itemize}

  \end{prop}

  \begin{proof}

    \begin {itemize}

    \item [(a)] Let $rm \in f^{-1}(P^{\prime})$ for some $r \in R$ and $m \in M$. This yields that $f(rm)=rf(m) \in P^{\prime}$. Since
                $P^{\prime}$ is an $S$-primary submodule of $M^{\prime}$, there is an $s \in S$ so that $sr \in \sqrt{(P^{\prime}:_{R}M^{\prime})}$ or $sf(m) \in P^{\prime}$. Now we will show that $(P^{\prime}:_{R}M^{\prime}) \subseteq (f^{-1}(P^{\prime}):_{R}M)$. Take $x \in (P^{\prime}:_{R}M^{\prime})$. Then we have $xM^{\prime} \subseteq P^{\prime}$. Since $f(M) \subseteq M^{\prime}$, we conclude that $f(xM)=xf(M) \subseteq xM^{\prime} \subseteq P^{\prime}$. This implies that $xM \subseteq f^{-1}(f(M)) \subseteq f^{-1}(p^{\prime})$ and thus $x \in (f^{-1}(P^{\prime}):_{R}M)$. As  $\sqrt{(P^{\prime}:_{R}M^{\prime})} \subseteq \sqrt{(f^{-1}(P^{\prime}):_{R}M)}$, we can conclude either $sa \in \sqrt{(f^{-1}(P^{\prime}):_{R}M)}$ or $sm \in f^{-1}(P^{\prime})$. Hence, $f^{-1}(P^{\prime})$ is an $S$-primary submodule of $M$.

    \item [(b)] First note that $(f(P):_{R}M) \cap S=\emptyset$. Otherwise there would be an $s \in (f(P):_{R}M^{\prime}) \cap S$. Since $s \in
                (f(P):_{R}M^{\prime})$, $sM^{\prime} \subseteq f(P)$, but then $f(sM)=sf(M)=sM^{\prime} \subseteq f(P)$. By taking their inverse images under $f$, we have
                $$sM \subseteq sM+ker(f) \subseteq f^{-1}(f(P))=P+ker(f)=P.$$
                That means $s \in (P:_{R}M)$, which contradictions $P$ is an $S$-primary submodule of $M$. Now take $r \in R$ and $m^{\prime} \in M^{\prime}$ with $rm^{\prime} \in f(P)$. As $f$ is an epimorphism, there is an $m \in M$ such that $m^{\prime}=f(m)$. Then $rm^{\prime}=rf(m)=f(rm) \in f(P)$. Since $Ker(f)$ is a subset of $P$, we get $rm \in P$. As $P$ is an $S$-primary submodule of $M$, there is an $s \in S$ so that $sr \in \sqrt{(P:_{R}M)}$ or $sm \in P$. Since $\sqrt{(P:_{R}M)} \subseteq \sqrt{(f(P):_{R}M^{\prime})}$, we have $sr \in \sqrt{(f(P):_{R}M^{\prime})}$ or $f(sm)=sf(m)=sm^{\prime} \in f(P)$. Accordingly, $f(P)$ is an $S$-primary submodule of $M^{\prime}$.

    \end {itemize}

  \end{proof}

  \begin{cor} \label{7}
                 Let $S$ be a multiplicatively closed subset of $R$ and take a submodule $L$ of $M$. Then we have the following.

         \begin {itemize}

                 \item [(a)] If $P^{\prime}$ is an $S$-primary submodule of $M$ with $(P^{\prime}:_{R}L) \cap S= \emptyset$,
                 then $L \cap P^{\prime}$ is an $S$-primary submodule of $L$.

                 \item [(b)] Suppose that $P$ is a submodule of $M$ with $L \subseteq P$. Then $P$ is an $S$-primary submodule of $M$ if and only if $P/L$ is an $S$-primary submodule of $M/L$.

          \end {itemize}

  \end{cor}

\begin{proof}
        \begin {itemize}

                 \item [(a)] Consider the injection $i:L \rightarrow M$ defined by $i(m)=m$ for all $m \in L$. Then note that $i^{-1}(P^{\prime})=L \cap P^{\prime}$. Now we will show that $(i^{-1}(P^{\prime}):_{R}L) \cap S= \emptyset$. Assume that $s \in (i^{-1}(P^{\prime}):_{R}L) \cap S$. Then we have $sL \subseteq i^{-1}(P^{\prime})=L \cap P^{\prime} \subseteq P^{\prime}$. This implies that $s \in (P^{\prime}:_{R}L) \cap S$, a contradiction. The rest follows from Proposition \ref{6.1} (a).

                 \item [(b)] Assume that $P$ is an $S$-primary submodule of $M$. Then consider the canonical homomorphism $\pi:M \rightarrow M/L$ defined by $\pi(m)=m+L$ for all $m \in M$. By Proposition \ref{6.1} (b), $P/L$ is an $S$-primary submodule of $M/L$. Conversely, assume that $P/L$ is an $S$-primary submodule of $M/L$. Let $rm \in P$ for some $r \in R$ and $m \in M$. This yields that $r(m+L)=rm+L \in P/L$. As $P/L$ is an $S$-primary submodule of $M/L$, there is an $s \in S$ so that $sr \in \sqrt{(P/L:_{R}M/L)}=\sqrt{(P:_{R}M)}$ or $s(m+L)=sm+L \in P/L$. Therefore, we have $sr \in \sqrt{(P:_{R}M)}$ or $sm \in P$. Hence, $P$ is an $S$-primary submodule of $M$.

        \end {itemize}
\end{proof}

           An $R$-module $M$ is said to be a multiplication module if for every submodule $N$ of $M$ there exists an ideal $I$ of $R$ such that $N=IM$ \cite{el1988multiplication}.

  \begin{prop} \label{8}
                 Let $M$ be an $R$-module and $S$ be a multiplicatively closed subset of $R$. The following statements hold.

  \begin {itemize}

    \item [(a)] If $P$ is an $S$-primary submodule of $M$, then $(P:_{R}M)$ is an $S$-primary ideal of $R$.

    \item [(b)] If $M$ is multiplication module and $(P:_{R}M)$ is an $S$-primary ideal of $R$, then $P$ is an $S$-primary submodule of $M$.

    \end {itemize}

  \end{prop}

       \begin{proof}
       	
       	 \begin {itemize}

       	     \item[(a)] Let $xy \in (P:_{R}M)$ for some $x, y \in R$. Then $xym \in P$ for all $m \in M$. As $P$ is an $S$-primary submodule, there exists $s \in S$ such that $sx \in \sqrt{(P:_{R}M)}$ or $sym \in P$ for all $m \in M$. If $sx \in \sqrt{(P:_{R}M)}$, there is nothing to prove. Suppose that $sx \notin \sqrt{(P:_{R}M)}$. Then $sym \in P$ for all $m \in M$ so that $sy \in (P:_{R}M)$. Therefore, $(P:_{R}M)$ is an $S$-primary ideal of $R$.

             \item[(b)] Let $J$ be an ideal of $R$ and $N$ a submodule of $M$ with $JN \subseteq P$. Then we can conclude that $J(N:_{R}M) \subseteq (JN:_{R}M) \subseteq (P:_{R}M)$. As $(P:_{R}M)$ is an $S$-primary ideal of $R$, there is an $s \in S$ so that $s(N:_{R}M) \subseteq (P:_{R}M)$ or $sJ \subseteq \sqrt{(P:_{R}M)}$. Thus, we can conclude that $sN=s(N:_{R}M)M \subseteq (P:_{R}M)M=P$ or $sJ \subseteq \sqrt{(P:_{R}M)}$. Therefore, by Theorem \ref{1} (d), $P$ is an $S$-primary submodule of $M$.

          \end {itemize}
  \end{proof}

\begin{rem}\label{9}

     \begin {itemize}

            \item[(a)] Assume that $M$ is a multiplication $R$-module and $K, L$ are two submodules of $M$. The product of $K$ and $L$ is defined as $KL=(K:_{R}M)(L:_{R}M)M$ \cite{ameri2003prime}.

            \item[(b)] Let $M$ be an $R$-module and $N$ a submodule of $M$. The radical of $N$, denoted by $rad(N)$, is the
            intersection of all prime submodules of $M$ containing $N$; that is,
            $rad(N)=\bigcap \{ P\,\, |\,\,  N \subseteq P ,\,\, P \in Spec(M) \}$ \cite{mccasland1986radicals}.

\end {itemize}
	
\end{rem}

As an immediate consequence of the Proposition \ref{8} and Theorem \ref{1} (d), we have the following explicit result.

  \begin{cor} \label{10}
              Suppose that $M$ is a multiplication $R$-module and $P$ a submodule of $M$ provided that $(P:_{R}M) \cap S=\emptyset$, where $S$ is a multiplicatively closed subset of $R$. Then the following are equivalent:

              \begin {itemize}

              \item[(a)] $P$ is an $S$-primary submodule of $M$;

              \item[(b)] There exists $s \in S$ such that whenever $LN \subseteq P$, where $L$ and $N$ are submodules of $M$, then $s(L:_{R}M) \subseteq \sqrt{(P:_{R}M)}$ or $sN \subseteq P$.

              \end {itemize}
  \end{cor}

  \begin{cor} \label{11}

        Suppose that $M$ is a finitely generated multiplication $R$-module and $P$ is a submodule of $M$ provided that $(P:_{R}M) \cap S=\emptyset$, where $S$ is a multiplicatively closed subset of $R$. Then the following are equivalent:

        \begin {itemize}

        \item[(a)] $P$ is an $S$-primary submodule of $M$;

        \item[(b)] There exists $s \in S$ such that whenever $LN \subseteq P$, where $L$ and $N$ are submodules of $M$, then $sL \subseteq rad(P)$ or $sN \subseteq P$.

        \end {itemize}

  \end{cor}

\begin{proof}
     (a)$\Rightarrow$(b). Assume that $LN \subseteq P$, where $L$ and $N$ are submodules of $M$. By Remark \ref{9} (a), $LN=(L:_{R}M)N \subseteq P$. Then there exists $s \in S$ so that $s(L:_{R}M) \subseteq \sqrt{(P:_{R}M)}$ or $sN \subseteq P$ by Theorem \ref{1} (d). Since $M$ is multiplication, by \cite[Theorem 2.12]{el1988multiplication}, we have $s(L:_{R}M)M=sL \subseteq \sqrt{(P:_{R}M)}M=rad(P)$ or $sN \subseteq P$.

     (b)$\Rightarrow$(a). Assume that $JN \subseteq P$, where $N$ is a submodule of $M$ and $J$ is an ideal of $R$. Set $K:=JM$. As $M$ is a multiplication module, Then we have
     $$KN=(K:_{R}M)(N:_{R}M)M=J(N:_{R}M)M=JN \subseteq P.$$
     By assumption, there exists $s \in S$ so that $sK \subseteq rad(P)$ or $sN \subseteq P$. As $M$ is finitely generated, by \cite[Thoerem 4.4]{mccasland1992prime}, $sK \subseteq rad(P)$ implies that
     $$sJ \subseteq (sK:_{R}M) \subseteq (rad(P):_{R}M)= \sqrt{(P:_{R}M)}.$$
     Therefore $P$ is an $S$-primary submodule of $M$ by Corollary \ref{10}.

  \end{proof}

  \begin{rem} \label{12}

          \begin {itemize}

          \item[(a)] Let $M$ be an $R$-module and $p$ be a maximal ideal of $R$. In \cite{el1988multiplication}, $T_{p}(M)$ is defined as follows
                       $$T_{p}(M) = \{\,\, m \in M |\,\,\, (1-r)m =0 \,\, \text{for some} \,\,\, r \in P\}.$$
          	Clearly $T_{p}(M)$ is a submodule of $M$. An $R$-module $M$ is said to be $p$-\textit{cyclic} provided there exist $q \in p$ and $m \in M$ such that $(1-q)M \subseteq Rm$ \cite{el1988multiplication}.

          \item[(b)] Let $M$ be an $R$-module. Then $M$ is a multiplication $R$-module if and only if for every maximal ideal $p$ of $R$ either $M = T_{p}(M)$ or $M$ is $p$-cyclic \cite[Theorem 1.2]{el1988multiplication}.

          \end {itemize}

  \end{rem}

 \begin{lem}\label{13}

        Let $S$ be a multiplicatively closed subset of $R$, $p$ be an $S$-primary (resp. $S$-prime) ideal of $R$ and $M$ be a faithful multiplication $R$-module. Then there exists a fixed $s \in S$ and whenever $am \in pM$, where $a \in R$ and $m \in M$, implies that $sa \in \sqrt{p}$ (resp. $sa \in p $) or $sm \in pM$.

\end{lem}

\begin{proof}
    It is enough to prove it for $S$-primary submodules. The technique is similar for $S$-prime.
	As $p$ is an $S$-primary ideal, there exists $s \in S$, whenever $rr^{\prime} \in p$, where $r, r^{\prime} \in R$, then $sr \in \sqrt{p}$ or $sr^{\prime} \in p$. Let $a \in R$ and $m \in M$ satisfy $am \in pM$. Suppose $sa \notin \sqrt{p}$. Set $K:=(pM:_{R}sm)$. Assume that $K\neq R$. Then there exists a maximal ideal $Q$ of $R$ so that $K \subseteq Q$. $m \notin T_{Q}(M)$, since otherwise, there exists $q \in Q$ such that $(1-q)m=0$ and so $(1-q)sm=0$. This implies that $(1-q) \in K \subseteq Q$, a contradiction. By \cite[Theorem 1.2]{el1988multiplication} $M$ is $Q$-cyclic, that is there exist $m^{\prime}\in M$ and $q \in Q$ such that $(1-q)M \subseteq Rm^{\prime}$. In particular, $(1-q)m=s^{\prime}m^{\prime}$, $(1-q)am=p^{\prime}m^{\prime}$ for some $s^{\prime} \in R$ and $p^{\prime} \in p$. Thus $(as^{\prime}-p^{\prime})m^{\prime}=0$. Now $(1-q)(Ann_{R}(m^{\prime}))M \subseteq (Ann_{R}(m^{\prime}))Rm^{\prime}=\mathbf{0}$ implies $(1-q)Ann_{R}(m) \subseteq Ann_{R}(M)=\mathbf{0}$, because $M$ is faithful, and hence $(1-q)as^{\prime}=(1-q)p^{\prime} \in p$. As $p$ is an $S$-primary ideal, $ss^{\prime} \in p$ or $sa \in \sqrt{p}$ or $s(1-q)^{n} \in p$ for some $n \in \mathbb{N}$.  But $p \subseteq K \subseteq Q$ so that in each case, we have a contradiction. It follows that $K=R$ and $sm \in pM$, as required.
\end{proof}

                In the following Theorem, The Theorem 2.11 in \cite{sevim2019s} will be extended by removing the condition ``finitely generated".

  \begin{thm} \label{14}
                Let $M$ be a multiplication $R$-module and $P$ a submodule of $M$ provided that $(P:_{R}M) \cap S= \emptyset$, where $S$ is a multiplicatively closed subset of $R$. Then the following are equivalent:

  \begin {itemize}

    \item [(a)] $P$ is an $S$-primary (resp. $S$-prime) submodule of $M$.

    \item [(b)] $(P:_{R}M)$ is an $S$-primary (resp. $S$-prime) ideal of $R$.

    \item [(c)] $P=IM$ for some $S$-primary (resp. $S$-prime) ideal $I$ of $R$ with $Ann(M) \subseteq I$.

    \end {itemize}

  \end{thm}

  \begin{proof}
        (a)$\Rightarrow$(b). It is clear from Proposition \ref{8} (a).

        (b)$\Rightarrow$(c). It is clear.

        (c)$\Rightarrow$(a). As $M$ is a faithful multiplication $R/Ann_{R}(M)$-module, by Corollary \ref{7} (b), $I/Ann_{R}(M)$ is an $S$-primary (resp. $S$-prime) ideal of $R/Ann_{R}(M)$. Hence $P=IM$ is an $S$-primary (resp. $S$-prime) submodule of $R/Ann_{R}(M)$-module $M$ by Lemma \ref{13}. Therefore, $P$ is an $S$-primary (resp. $S$-prime) submodule of $R$-module $M$, as required.

  \end{proof}

  \begin{prop} \label{15}

       Let $P$ be an $S$-primary submodule of multiplication $R$-module $M$. Suppose that $N \cap L \subseteq P$ for some submodules $N$ and $L$ of $M$. Then $sN \subseteq P$ or $sL \subseteq rad(P)$ for some $s \in S$.

  \end{prop}

  \begin{proof}
    Since $P$ is an $S$-primary submodule, there exists $s \in S$ such that for every $r \in R$ and $m \in M$, if $rm \in P$, then $sr \in \sqrt{(P:_{R}M)} $ or $sm \in P$. Let $sN \nsubseteq P$. Then $sm^{\prime} \notin P$ for some $m^{\prime} \in N$. Take an element $a \in (L:_{R}M)$. This yields that $am^{\prime} \in (L:_{R}M)N \subseteq L \cap N \subseteq P$. As $P$ is an $S$-primary submodule of $M$ and $sm^{\prime} \notin P$, we can conclude that $sa \in \sqrt{(P:_{R}M)}$ so that $s(L:_{R}M) \subseteq \sqrt{(P:_{R}M)}$. As $M$ is a multiplication module, by \cite[Theorem 2.12]{el1988multiplication} we have
    $$sL=s(L:_{R}M)M \subseteq \sqrt{(P:_{R}M)}M=rad(P).$$

  \end{proof}

     \begin{lem}\label{16}

       Let $R=R_{1}\times R_{2}$ and $S=S_{1} \times S_{2}$ where $S_{i}$ is a multiplicatively closed subset of $R_{i}$. Suppose $p=p_{1} \times p_{2}$ is an ideal of $R$. So the following are equivalent:

       \begin {itemize}

       \item [(a)] $p$ is an $S$-primary ideal of $R$.

       \item [(b)] $p_{1}$ is an $S_{1}$-primary ideal of $R_{1}$ and $p_{2} \cap S_{2}\neq \emptyset$ or $p_{2}$ is an $S_{2}$-primary ideal of $R_{2}$ and $p_{1} \cap S_{1}\neq \emptyset$.

       \end {itemize}

     \end{lem}

  \begin{proof}
 	(a)$\Rightarrow$(b). Since $(1,0)(0,1)=(0,0) \in p$, there exists $s=(s_{1},s_{2}) \in S$ so that $s(1,0)=(s_{1},0) \in \sqrt{p}$ or $s(0,1)=(0,s_{2}) \in p$ and thus $p_{1} \cap S_{1} \neq \emptyset$ or $p_{2} \cap S_{2}\neq \emptyset$. We may assume that $p_{1} \cap S_{1} \neq \emptyset$. As $P \cap S=\emptyset$, we have $p_{2} \cap S_{2}=\emptyset$. Let $xy \in p_{2}$ for some $x,y \in R_{2}$. Since $(0,x)(0,y) \in p$ and $p$ is an $S$-primary ideal of $R$. We get either $s(0,x)=(0,s_{2}x) \in \sqrt{p}$ or $s(0,y)=(0,s_{2}y) \in p$ and this yields $s_{2}x \in \sqrt{p_{2}}$ or $s_{2}y \in p_{2}$. Therefore, $p_{2}$ is an $S$-primary ideal of $R_{2}$. In the other case, one can easily show that $p_{1}$ is an $S$-primary ideal of $R_{1}$.
 	
 	(b)$\Rightarrow$(a). Assume that $p_{1} \cap S_{1}\neq \emptyset$ and $p_{2}$ is an $S$-primary ideal of $R_{2}$. Then there exists $s_{1} \in p_{1} \cap S_{1}$. Let $(a,b)(c,d)=(ac,bd) \in p$ for some $a, c \in R_{1}$ and $b,d \in R_{2}$. This yields that $bd \in p_{2}$ and thus there exists $s_{2} \in S_{2}$ so that $s_{2}b \in \sqrt{p_{2}}$ or $s_{2}d \in p_{2}$. Put $s=(s_{1},s_{2}) \in S$. Then note that $s(a,b)=(s_{1}a,s_{2}b) \in \sqrt{p}$ or $s(c,d) \in p$. Therefore, $p$ is an $S$-primary ideal of $R$. In other case, one can similarly prove that $p$ is an $S$-primary ideal of $R$.
 	
 \end{proof}

 \begin{thm}\label{17}
	
	Suppose that $M=M_{1}\times M_{2}$ and $R=R_{1} \times R_{2}$-module and $S=S_{1} \times S_{2}$ is a multiplicatively closed subset of $R$, where $M_{i}$ is a $R_{i}$-module and  $S_{i}$ is a multiplicatively closed subset of $R_{i}$ for each $i=1, 2$. Assume $P=P_{1} \times P_{2}$ is a submodule of $M$. Then the following are equivalent:
	
	\begin {itemize}
	
	\item [(a)] $P$ is an  $S$-primary submodule of $M$.
	
	\item [(b)] $P_{1}$ is an $S_{1}$-primary submodule of $M_{1}$ and $(P_{2}:_{R_{2}}M_{2}) \cap S_{2}\neq \emptyset$ or $P_{2}$ is an $S_{2}$-primary submodule of $M_{2}$ and $(P_{1}:_{R_{1}}M_{1}) \cap S_{1}\neq \emptyset$.
	
	\end {itemize}
	
\end{thm}

\begin{proof}
	(a)$\Rightarrow$(b). By Proposition \ref{8}, $(P:_{R}M)=(P_{1}:_{R_{1}}M_{1})\times (P_{2}:_{R_{2}}M_{2})$ is an $S$-primary ideal of $R$ and so by Lemma \ref{16}, either $(P_{1}:_{R_{1}}M_{1}) \cap S_{1}\neq \emptyset$ or $(P_{2}:_{R_{2}}M_{2}) \cap S_{2}\neq \emptyset$. We may assume that $(P_{1}:_{R_{1}}M_{1}) \cap S_{1}\neq \emptyset$. Now we will show that $P_{2}$ is an $S_{2}$-primary submodule of $M_{2}$. Let $rm \in P_{2}$ for some $r \in R_{2}$ and $m \in M_{2}$. Then $(1,r)(0,m)=(0,rm) \in P$. As $P$ is an $S$-primary, there is an $s=(s_{1},s_{2}) \in S$ so that $s(1,r)=(s_{1},s_{2}r) \in \sqrt{(P:_{R}M)}$ or $s(0:_{R}m)=(0,s_{2}m) \in P$. This implies that $s_{2}r \in \sqrt{(P_{2}:_{R_{2}}M_{2})}$ or $s_{2}m \in P_{2}$. Therefore, $P_{2}$ is an $S_{2}$ is an $S_{2}$-primary submodule of $M_{2}$. In the other case, it can be similarly show that $P_{1}$ is an $S_{1}$-primary submodule of $M_{1}$.
	
	(b)$\Rightarrow$(a). Assume that $(P_{1}:_{R_{1}}M_{1}) \cap S_{1}\neq \emptyset$ and $P_{2}$ is an $S_{2}$-primary submodule of $M_{2}$. Then there exists $s_{1} \in (P_{1}:_{R_{1}}M_{1}) \cap S_{1}$. Let $(r_{1},r_{2})(m_{1},m_{2})=(r_{1}m_{1},r_{2}m_{2}) \in P$ for some $r_{i} \in R_{i}$ and $m_{i} \in M_{i}$, where $i=1,2$. Then $r_{2}m_{2} \in P_{2}$. As $P_{2}$ is an $S_{2}$-primary submodule of $M_{2}$, there is an $s_{2} \in S_{2}$ so that $s_{2}r_{2} \in \sqrt{(P_{2}:_{R_{2}}M_{2})}$ or $s_{2}m_{2} \in P_{2}$. Now put $s=(s_{1},s_{2}) \in S$. Then note that $s(r_{1},r_{2})=(s_{1}r_{1},s_{2}r_{2}) \in  \sqrt{(P:_{R}M)}$ or $s(m_{1},m_{2})=(s_{1}m_{1},s_{2}m_{2}) \in P_{1}\times P_{2}=P$. Therefore, $P$ is an $S$-primary submodule of $M$. Similarly one can show that if $P_{1}$ is an $S_{1}$-primary submodule of $M_{1}$ and $(P_{2}:_{R_{2}}M_{2}) \cap S_{2}\neq \emptyset$, then $P$ is an $S$-primary submodule of $M$.
	
\end{proof}

 \begin{thm}\label{18}
	
	Let $M=M_{1}\times M_{2} \times \cdots M_{n}$ and $R=R_{1} \times R_{2} \times \cdots R_{n}$-module and $S=S_{1} \times S_{2} \times \cdots S_{n}$ is a multiplicatively closed subset of $R$, where $M_{i}$ is a $R_{i}$-module and  $S_{i}$ is a multiplicatively closed subset of $R_{i}$ for each $i=1, 2, \cdots, n$. Assume $P=P_{1} \times P_{2} \times \cdots P_{n}$ is a submodule of $M$. Then the following are equivalent:
	
	\begin {itemize}
	
	\item [(a)] $P$ is an  $S$-primary submodule of $M$.
	
	\item [(b)] $P_{i}$ is an $S_{i}$-primary submodule of $M_{i}$ for some $i\in \{1, 2, \cdots, n\}$ and $(P_{j}:_{R_{j}}M_{j}) \cap S_{j}\neq \emptyset$ for all $j \in \{1,2, \cdots, n\}-\{i\}$.
	
	\end {itemize}
	
\end{thm}

\begin{proof}
	We apply induction on $n$. For $n=1$, the result is true. If $n=2$, then (a)$\Leftrightarrow$(b) follows from Theorem \ref{17}. Assume that (a) and (b) are equivalent when $k<n$. Now, we shall prove (a)$\Leftrightarrow$(b) when $k=n$. Let $P=P_{1}\times P_{2}\times ... \times P_{n}$. Put $P^{\prime}=P_{1}\times P_{2}\times ... \times P_{n-1}$ and $S^{\prime}=S_{1}\times S_{2}\times ... \times S_{n-1}$. Then by Theorem \ref{17}, the necessary and sufficient condition for $P=P^{\prime}\times P_{n}$ is an $S$-primary submodule of $M$ is that $P^{\prime}$ is an $S$-primary submodule of $M^{\prime}$ and $(P_{n}:_{R_{n}}M_{n}) \cap S_{n}\neq \emptyset$ or $P_{n}$ is an $S$-primary submodule of $M_{n}$ and $(P^{\prime}:_{R^{\prime}}M^{\prime}) \cap S^{\prime}\neq \emptyset$, where $M^{\prime}=M_{1}\times M_{2}\times ... \times M_{n-1}$ and $R^{\prime}=R_{1}\times R_{2}\times ... \times R_{n-1}$. The rest follows from the induction hypothesis.
	
\end{proof}

\begin{lem}\label{19}
	
	Suppose that $P$ is an $S$-primary submodule of $M$. Then the following statements hold for some $s \in S$.
	
	\begin {itemize}
	
	\item [(a)] $(P:_{M}s^{\prime}) \subseteq (P:_{M}s)$ for all $s^{\prime} \in S$.
	
	\item [(b)]  $((P:_{R}M):_{R}s^{\prime}) \subseteq ((P:_{R}M):_{R}s)$ for all $s^{\prime} \in S$
	
	\end {itemize}
	
\end{lem}

\begin{proof}
	\begin {itemize}
	
	\item [(a)] Take an element $m^{\prime} \in (P:_{M}s^{\prime})$, where $s^{\prime} \in S$. Then $s^{\prime}m^{\prime} \in P$. Since $P$ is an $S$-primary submodule of $M$, there exists $s \in S$ such that $ss^{\prime} \in  \sqrt{(P:_{R}M)}$ or $sm^{\prime} \in P$. As $(P:_{R}M) \cap S=\emptyset$, we get $sm^{\prime} \in P$, namely $m^{\prime} \in (P:_{M}s)$.
	
	\item [(b)] Follows from part (a).
	
	\end {itemize}
	
\end{proof}

\begin{prop}\label{20}
	
	Suppose that $M$ is a finitely generated $R$-module, $S$ is a multiplicatively closed subset of $R$, and $P$ is a submodule of $M$ satisfying $(P:_{R}M) \cap S= \emptyset$ . Then the following are equivalent:
	
	\begin {itemize}
	
	\item [(a)] $P$ is an $S$-primary submodule of $M$.
	
	\item [(b)]  $S^{-1}P$ is a primary submodule of $S^{-1}M$ and there is an $s \in S$  satisfying $(P:_{M}s^{\prime}) \subseteq (P:_{M}s)$ for all $s^{\prime} \in S$.
	
	\end {itemize}
	
\end{prop}

\begin{proof}
	(a)$\Rightarrow$(b). It is clear from Proposition \ref{4} (c) and Lemma \ref{19}.
	
	(b)$\Rightarrow$(a). Take $a \in R$ and $m \in M$ with $am \in P$. Then $\frac{a}{1}.\frac{m}{1} \in S^{-1}P$. Since $S^{-1}P$ is a primary submodule of $S^{-1}M$ and $M$ is finitely generated, we can conclude that $\frac{a}{1} \in \sqrt{(S^{-1}P:_{S^{-1}R}S^{-1}M)}=\sqrt{S^{-1}(P:_{R}M)}$ or $\frac{m}{1} \in S^{-1}P$. Then $ua \in \sqrt{(P:_{R}M)}$ or $u^{\prime}m \in P$ for some $u, u^{\prime} \in S$. By assumption, there is an $s \in S$ so that $(P:_{R}s^{\prime}) \subseteq (P:_{R}s)$ for all $s^{\prime} \in S$. If $ua \in \sqrt{(P:_{R}M)}$, then $a^{n}M \subseteq (P:_{M}u^{n}) \subseteq (P:_{R}s)$ for some $n \in \mathbb{N}$ and thus $sa \in \sqrt{(P:_{R}M)}$. If $u^{\prime}m \in P$, a similar argument shows that $sm \in P$. Therefore, $P$ is an $S$-primary submodule of $M$.
	
\end{proof}

 \begin{thm}\label{21}
	
	    Suppose that $P$ is a submodule of $M$ provided $(P:_{R}M) \cap S= \emptyset$. Then $P$ is an $S$-primary submodule of $M$ if and only if $(P:_{M}s)$ is a primary submodule of $M$ for some $s \in S$.
	
\end{thm}

\begin{proof}
	Assume $(P:_{M}s)$ is a primary submodule of $M$ for some $s \in S$. Let $am \in P$, where $a \in R$ and $m \in M$. As $am \in (P:_{M}s)$, we get $a \in \sqrt{((P:_{M}s):_{R}M)}$ or $m \in (P:_{M}s)$. This yields that $as \in \sqrt{(P:_{R}M)}$ or $sm \in P$. Conversely, assume that $P$ is an $S$-primary submodule of $M$. Then there exists $s \in S$ such that whenever $am \in P$, where $a \in R$ and $m \in M$, then $sa \in \sqrt{(P:_{R}M)}$ or $sm \in P$. Now we prove that $(P:_{M}s)$ is primary. Take $r \in R$ and $m \in M$ with $rm \in (P:_{M}s)$. Then $srm \in P$. As $P$ is $S$-primary, we get $s^{2}r \in \sqrt{(P:_{R}M)}$ or $sm \in P$. If $sm \in P$, then there is nothing to show. Assume that $sm \notin P$. Then $s^{2}r \in \sqrt{(P:_{R}M)}$ and hence $sr \in \sqrt{(P:_{R}M)}$. Thus $r^{n} \in ((P:_{R}M):_{R}s^{n}) \subseteq ((P:_{R}M):_{R}s)$ for some $n \in \mathbb{N}$, by Lemma \ref{19}. Thus, we can conclude that $r^{n} \in  ((P:_{M}s):_{R}M)$, namely $r \in \sqrt{((P:_{M}s):_{R}M)}$. Hence $(P:_{M}s)$ is a prime submodule of $M$.
\end{proof}

 \begin{thm}\label{22}
	
	Suppose that $P$ is a submodule of $M$ provided $(P:_{R}M) \subseteq Jac(R)$, where $Jac(R)$ is the Jacobson radical of $R$. Then the following statements are equivalent:
	
	\begin {itemize}
	
	\item [(a)] $P$ is a primary submodule of $M$.
	
	\item [(b)] $(P:_{R}M)$ is a primary ideal of $R$ and $P$ is an $(R- \mathfrak{m})$-primary submodule of $M$ for each $\mathfrak{m} \in Max(R)$.
	
	\end {itemize}
	
\end{thm}

\begin{proof}
	(a)$\Rightarrow$(b). Since $(P:_{R}M) \subseteq Jac(R)$, $(P:_{R}M) \subseteq \mathfrak{m}$ for each $\mathfrak{m} \in Max(R)$ and hence $(P:_{R}M) \cap (R-\mathfrak{m})=\emptyset$. The rest follows from Lemma \ref{0.3} (a).
	
	(b)$\Rightarrow$(a). Let $am \in P$ with $a \notin (P:_{R}M)$ for some $a \in R$ and $m \in M$. Let $\mathfrak{m} \in Max(R)$. As $P$ is an $(R-\mathfrak{m})$-primary submodule of $M$, there exists $s_{\mathfrak{m}} \notin \mathfrak{m}$ such that $as_{\mathfrak{m}} \in \sqrt{(P:_{R}M)}$ or $s_{\mathfrak{m}}m \in P$. As $(P:_{R}M)$ is a primary ideal of $R$ and $s_{\mathfrak{m}} \notin \sqrt{(P:_{R}M)}$, we have $as_{\mathfrak{m}} \notin (P:_{R}M)$ and so $s_{\mathfrak{m}}m \in P$. Now consider the set $\Omega=\{ s_{\mathfrak{m}} \,|\, \exists \,\, \mathfrak{m} \in Max(R),\, s_{\mathfrak{m}} \notin \mathfrak{m}\,\, \text{and} \,\, s_{\mathfrak{m}}m \in P \}$. Then note that $(\Omega)=R$. To see this, take any maximal ideal $\mathfrak{m}^{\prime}$ containing $\Omega$. Then the definition of $\Omega$ requires that there exists $s_{\mathfrak{m}^{\prime}} \in \Omega$ and $s_{\mathfrak{m}^{\prime}} \notin \mathfrak{m}^{\prime}$. As $\Omega \subseteq \mathfrak{m}^{\prime}$, we have $s_{\mathfrak{m}^{\prime}} \in \Omega \subseteq \mathfrak{m}^{\prime}$, a contradiction. Thus $(\Omega)=R$, and this yields $1=r_{1}s_{\mathfrak{m}_{1}}+r_{2}s_{\mathfrak{m}_{2}}+\cdots+r_{n}s_{\mathfrak{m}_{n}}$ for some $r_{i} \in R$ and $s_{\mathfrak{m}_{i}} \notin \mathfrak{m}_{i}$ with $s_{\mathfrak{m}_{i}}m \in P$, where $\mathfrak{m}_{i} \in Max(R)$ for each $i=1,2,...,n$. This yields that $m=r_{1}s_{\mathfrak{m}_{1}}m+r_{2}s_{\mathfrak{m}_{2}}m+\cdots+r_{n}s_{\mathfrak{m}_{n}}m \in P$. Therefore, $P$ is a primary submodule of $M$.
	
\end{proof}

Now we determine all primary submodules of a module over a quasi-local ring in terms of $S$-primary submodules.

 \begin{cor}\label{23}
	
	Suppose $M$ is a module over a quasi-local ring $(R, \mathfrak{m})$. Then the following statements are equivalent:
	
	\begin {itemize}
	
	\item [(a)] $P$ is a primary submodule of $M$.
	
	\item [(b)] $(P:_{R}M)$ is a primary ideal of $R$ and $P$ is an $(R- \mathfrak{m})$-primary submodule of $M$ for each $\mathfrak{m} \in Max(R)$.
	
	\end {itemize}
	
\end{cor}

\begin{proof}
	It is clear from Theorem \ref{22}.
	
\end{proof}

\begin{rem}\label{24}
	
	\begin{itemize}
		
		\item [(a)] Suppose that $M$ is an $R$-module. The idealization $R(+)M\\=\{(a,m) \,\,| \, a \in R, \, m \in M \}$ of $M$ is a commutative ring whose addition is component-wise and whose multiplication is defined as $(a,m)(b,m^{\prime})=(ab,am^{\prime}+bm)$ for each $a, \, b \in R$ and $m, \, m^{\prime} \in M$. If $S$ is a multiplicatively closed subset of $R$ and $P$ is a submodule of $M$, then $S(+)P=\{(s, p)\,|\, s \in S, \, p \in P\}$ is a multiplicatively closed subset of $R(+)M$ \cite{anderson2009idealization, nagata1962local}.
		
		\item [(b)] Radical ideals of $R(+)M$ have the form $I(+)M$, where $I$ is a radical ideal of $R$. If $J$ is an ideal of $R(+)M$, then $\sqrt{J}=\sqrt{I}(+)M$. In particular, if $I$ is an ideal of $R$ and $N$ is a submodule of $M$, then $\sqrt{I(+)N}=\sqrt{I}(+)M$ \cite[Theorem 3.2 (3)]{anderson2009idealization}.
		
	\end{itemize}

\end{rem}

\begin{prop}\label{24.1}
	Let $M$ be an $R$-module and $p$ be an ideal of $R$ such that $p \subseteq Ann(M)$. Then the following are equivalent:
	
	\begin{itemize}
		
		\item [(a)] $p$ is a primary ideal of $R$.
		
		\item [(b)] $p(+)M$ is a primary ideal of $R(+)M$.
		
	\end{itemize}

\end{prop}

\begin{proof}
	This is straightforward.
\end{proof}

\begin{thm}\label{25}
	Let $S$ be a multiplicatively closed subset of $R$, $p$ be an ideal of $R$ provided $p \cap S= \emptyset$ and $M$ be an $R$-module. Then the following are equivalent:
	
	\begin{itemize}
		
		\item [(a)] $p$ is an $S$-primary ideal of $R$.
		
		\item [(b)] $p(+)M$ is an $S(+)0$-primary ideal of $R(+)M$.
		
		\item [(c)] $p(+)M$ is an $S(+)M$-primary ideal of $R(+)M$.
		
	\end{itemize}
	
\end{thm}

\begin{proof}
	(a)$\Rightarrow$(b). Let $(x,m)(y,m^{\prime})=(xy,xm^{\prime}+ym) \in p(+)M$, where $x, y \in R$ and $m, m^{\prime} \in M$. Then we get $xy \in p$. As $p$ is $S$-primary, there exists $s \in S$ such that $sx \in \sqrt{p}$ or $sy \in p$. Now put $s^{\prime}=(s,0) \in S(+)0$. Then we have $s^{\prime}(x,m)=(sx,sm) \in \sqrt{p}(+)M=\sqrt{p(+)M}$ or $s^{\prime}(y,m^{\prime})=(sy,sm^{\prime}) \in p(+)M$. Therefore, $p(+)M$ is an $S(+)0$-primary ideal of $R(+)M$.
	
	(b)$\Rightarrow$(c). It is clear from Proposition \ref{4}.
	
	(c)$\Rightarrow$(a). Let $xy \in p$ for some $x,y \in R$. Then $(x,0)(y,0) \in p(+)M$. Since $p(+)M$ is $S(+)M$-primary, there exists $s=(s_{1},m_{1}) \in S(+)M$ such that $s(x,0)=(s_{1}x,xm_{1}) \in \sqrt{p(+)M}=\sqrt{p}(+)M$ or $s(y,0)=(s_{1}y,ym_{1}) \in p(+)M$ and hence we get $s_{1}x \in \sqrt{p}$ or $s_{1}y \in p$. Therefore $p$ is an $S$-primary ideal of $R$.
	
\end{proof}

\begin{rem}\label{26}
	Let $M$ be an $R$-module and let $S$ be a multiplicatively closed subset of $R$ such that $Ann_{R}(M) \cap S=\emptyset$. We say that $M$ is an $S$-\textit{torsion-free module} in the case that there is an $s \in S$ such that if $rm=0$, where $r \in R$ and $m \in M$, then $sm=0$ or $sr=0$ \cite[Definition 2.23]{sevim2019s}.

\end{rem}

\begin{prop}\label{27}
     Let $M$ be an $R$-module. Assume that $P$ is a submodule of $M$ and $S$ is a multiplicatively closed subset of $R$ such that $Ann_{R}(M) \cap S=\emptyset$. Then $P$ is an $S$-primary submodule of $M$ if and only if the factor module $M/P$ is a $\pi(S)$-torsion-free $R/\sqrt{(P:_{R}M)}$-module, where $\pi:R \rightarrow R/\sqrt{(P:_{R}M)}$ is the canonical homomorphism.
	
\end{prop}

\begin{proof}
	Suppose that $P$ is an $S$-primary submodule of $M$. Let $\overline{am}=0_{M/P}$, where $\overline{a}=a+\sqrt{(P:_{R}M)}$ and $\overline{m}=m+P$ for some $a \in R$ and $m \in M$. This yields that $am \in P$. As $P$ is $S$-primary, there exists $s \in S$ such that $sa \in \sqrt{(P:_{R}M)}$ or $sm \in P$. Then we can conclude that $\pi(s)\overline{a}=0_{R/\sqrt{(P:_{R}M)}}$ or $\pi(s)\overline{m}=0_{M/P}$. Therefore, $M/P$ is a $\pi(S)$-torsion-free $R/\sqrt{(P:_{R}M)}$-module. For the other direction, suppose that $M/P$ is a $\pi(S)$-torsion-free $R/\sqrt{(P:_{R}M)}$-module. Let $am \in P$,where $a \in R$ and $m \in M$. Put $\overline{a}=a+\sqrt{(P:_{R}M)}$ and $\overline{m}=m+P$. Then note that $\overline{am}=0_{M/P}$. As $M/P$ is a $\pi(S)$-torsion-free $R/\sqrt{(P:_{R}M)}$-module, there exists $s \in S$ such that $\pi(s)\overline{a}=0_{R/\sqrt{(P:_{R}M)}}$ or $\pi(s)\overline{m}=0_{M/P}$. This yields that $sa \in \sqrt{(P:_{R}M)}$ or $sm \in P$. Accordingly, $P$ is an $S$-primary submodule of $M$.
\end{proof}

\begin{defn}\label{28}
	Let $M$ be an $R$-module and let $S$ be a multiplicatively closed subset of $R$ such that $Ann_{R}(M) \cap S=\emptyset$. We say that $M$ is a \textit{quasi} $S$-\textit{torsion-free module}, if there exists $s \in S$ such that whenever $rm=0$, where $r \in R$ and $m \in M$, then $sm=0$ or $(sr)^{t}=0$ for some $t \in \mathbb{N}$.
\end{defn}

According to Definition \ref{28}, Proposition \ref{27} can be expressed as follows.

\begin{prop}\label{29}
	Let $M$ be an $R$-module. Assume that $P$ is a submodule of $M$ and $S$ is a multiplicatively closed subset of $R$ such that $Ann_{R}(M) \cap S=\emptyset$. Then $P$ is an $S$-primary submodule of $M$ if and only if the factor module $M/P$ is a quasi $\pi^{\prime}(S)$-torsion-free $R/(P:_{R}M)$-module, where $\pi^{\prime}:R \rightarrow R/(P:_{R}M)$ is the canonical homomorphism.	
\end{prop}

\begin{thm}\label{31}
	Let $M$ be a module over an integral domain $R$. The following are equivalent:
	\begin{itemize}
		\item [(a)] $M$ is a torsion-free module;
		
		\item [(b)] $M$ is a quasi $(R-p)$-torsion-free module for each $p \in Spec(R)$;
		
		\item [(c)] $M$ is a quasi $(R-\mathfrak{m})$-torsion-free module for each $\mathfrak{m} \in Max(R)$.
			
	\end{itemize}
\end{thm}

\begin{proof}
	(a)$\Rightarrow$(b). It is clear.
	
	(b)$\Rightarrow$(c). It is clear.
	
	(c)$\Rightarrow$(a). Assume that $a\neq 0$. Take $\mathfrak{m} \in Max(R)$. As $M$ is quasi $(R-\mathfrak{m})$-torsion-free, there exists $s_{m} \neq \mathfrak{m}$ so that $s_{m}m=0$ or $(s_{m}a)^{t}=0$ for some $t \in \mathbb{N}$. As $R$ is an integral domain, $(s_{m}a)^{t}\neq0$. Now, put $\Omega=\{\, s_{m} \in R \, |\, \exists \mathfrak{m} \in Max(R), s_{m} \notin \mathfrak{m} \,\, \text{and} \,\, s_{m}m=0 \}$. A similar argument in the proof of Theorem \ref{22} shows that $\Omega=R$. Then we have $(s_{m_{1}})+(s_{m_{2}})+\cdots+(s_{m_{n}})=R$ for some $(s_{m_{i}}) \in \Omega$. this implies that $Rm=\sum_{i=1}^{n}(s_{m_{i}})m=(0)$ and hence $m=0$. This means $M$ is a torsion-free module.
	
\end{proof}




\bibliographystyle{plain}

\end{document}